\documentclass[english]{article}
\usepackage[T1]{fontenc}
\usepackage{geometry}
\usepackage{color}
\usepackage{babel}
\usepackage{refstyle}
\usepackage{float}
\usepackage{amsmath}
\usepackage{amsthm}
\usepackage{amssymb}
\usepackage{graphicx}
\usepackage{setspace}
\usepackage[authoryear]{natbib}
\usepackage[unicode=true,pdfusetitle,
 bookmarks=true,bookmarksnumbered=true,bookmarksopen=true,bookmarksopenlevel=1,
 breaklinks=false,pdfborder={0 0 1},backref=false,colorlinks=true]
 {hyperref}

\makeatletter

\providecommand{\tabularnewline}{\\}
\newcommand{\lyxdot}{.}

\RS@ifundefined{subsecref}
  {\newref{subsec}{name = \RSsectxt}}
  {}
\RS@ifundefined{thmref}
  {\def\RSthmtxt{theorem~}\newref{thm}{name = \RSthmtxt}}
  {}
\RS@ifundefined{lemref}
  {\def\RSlemtxt{lemma~}\newref{lem}{name = \RSlemtxt}}
  {}

\theoremstyle{definition}
    \ifx\thechapter\undefined
      \newtheorem{defn}{\protect\definitionname}
    \else
      \newtheorem{defn}{\protect\definitionname}[chapter]
    \fi
\theoremstyle{definition}
    \ifx\thechapter\undefined
      \newtheorem{example}{\protect\examplename}
    \else
      \newtheorem{example}{\protect\examplename}[chapter]
    \fi
\theoremstyle{plain}
    \ifx\thechapter\undefined
      \newtheorem{lem}{\protect\lemmaname}
    \else
      \newtheorem{lem}{\protect\lemmaname}[chapter]
    \fi
\theoremstyle{plain}
    \ifx\thechapter\undefined
      \newtheorem{prop}{\protect\propositionname}
    \else
      \newtheorem{prop}{\protect\propositionname}[chapter]
    \fi
\theoremstyle{remark}
    \ifx\thechapter\undefined
      \newtheorem{claim}{\protect\claimname}
    \else
      \newtheorem{claim}{\protect\claimname}[chapter]
    \fi

\usepackage{url}

\makeatother

\providecommand{\claimname}{Claim}
\providecommand{\definitionname}{Definition}
\providecommand{\examplename}{Example}
\providecommand{\lemmaname}{Lemma}
\providecommand{\propositionname}{Proposition}

\begin{document}
\title{Iteration of Functions $f:X^{k}\rightarrow X$ and their Periodicity}
\author{Suneil Parimoo\thanks{Columbia University. Email: sp3681@columbia.edu}}
\date{This version: October 25, 2020}
\maketitle
\begin{abstract}
We propose a notion of iterating functions $f:X^{k}\rightarrow X$
in a way that represents recurrence relations of the form $a_{n+k}=f(a_{n},a_{n+1},...,a_{n+k-1})$.
We define a function as \textit{$n$-involutory} when its $n$th iterate
is the identity map, and discuss elementary group-theoretic properties
of such functions along with their relation to cycles of their corresponding
recurrence relations. Further, it is shown that a function $f:X^{k}\rightarrow X$
that is $2$-involutory in each of its $k$ arguments (holding others
fixed) is $(k+1)$-involutory.

\noindent \textbf{Mathematics Subject Classification (2020):} Primary:
39B12; Secondary: 05E05, 20A05, 26A18, 30D05, 37B20, 37C25

\noindent \textbf{Keywords:} Iterative functional equations, iterative
roots, involution, periodic function, symmetric function, Babbage
equation, recurrence relation, difference equation, fixed point, cycle
\end{abstract}

\section{{\large{}\label{sec1}Introduction}}

This paper proposes a means to extend the notion of function iteration
to functions $f:X^{k}\rightarrow X$ for some set $X$ and integer
$k\geq1$, and aims to explore such functions that obey a certain
iterative periodic property.

Function iteration is well defined when functions are self-maps. For
a self-map $f:X\rightarrow X$, the $n$th iterate of $f$, denoted
by $f^{n}$ for some nonnegative integer $n$, is defined recursively
by 
\[
f^{0}\equiv id_{X}
\]
and 
\[
f^{n+1}\equiv f\circ f^{n},
\]
where $id_{X}$ is the identity map on $X$. If $f$ is invertible
with inverse $f^{-1}$, then this definition extends to negative iterates,
where the $-n$th iterate of $f$ is the $n$th iterate of $f^{-1}$.
The associativity of function composition immediately gives the following
properties for integers $m$ and $n$:

\begin{equation}
1.\text{ Addition rule: }f^{m}\circ f^{n}=f^{n}\circ f^{m}=f^{m+n}\label{eq:1}
\end{equation}
\begin{equation}
2.\text{ Multiplication rule: }\left(f^{m}\right)^{n}=\left(f^{m}\right)^{n}=f^{mn}.\label{eq:2}
\end{equation}
The multiplication rule (\ref{eq:2}) provides a natural means of
extension to fractional iterates; for integers $n\neq0$ and $m$,
with $gcd(m,n)=1$, an $\frac{m}{n}$th iterate of $f$ is any function
$g$ such that $g^{n}=f^{m}$ (c.f. \citet{Isaacs1950}).

Function iteration for self-maps may be understood as a representation
of a recurrence relation. For a sequence $\{a_{n}\}_{n\in\mathbb{N}}$
defined recursively through some self-map $f$ by 
\[
a_{n+1}=f(a_{n})
\]
given some seed value $a_{0}$, the $n$th term of the sequence may
be computed as 
\[
a_{n}=f^{n}(a_{0}).
\]
In extending the notion of function iteration to a function $f:X^{k}\rightarrow X$,
it is desirable for the iterates to likewise represent the state of
a recurrence relation. One way to define its iterates is by defining
its first iterate, $f^{1}$, as the self-map on $X^{k}$ given by
\begin{equation}
f^{1}:(x_{1},x_{2},...,x_{k})\rightarrow(\underbrace{f(x_{1},x_{2},...,x_{k}),...,f(x_{1},x_{2},...,x_{k}}_{k\text{ times}}),\label{eq:2-2}
\end{equation}
and defining other iterates of $f$ as typical iterates of $f^{1}.$
Iterates thus defined treat the $k$ arguments in a symmetric manner
and consequently feature redundancies in that each of their $k$ component
functions are identical. However, it is not immediately clear what
distinct application this definition of function iteration serves.\footnote{Defining function iteration as in (\ref{eq:2-2}) can be understood
as representing a recurrence relation $a_{n+1}=f(\underbrace{a_{n},...,a_{n}}_{k\text{ times}})$.
However, such a system can be written succinctly as $a_{n+1}=g(a_{n})$
for self-map $g$ and thus can be represented by usual function iteration. } Instead, the definition of function iteration for a function $f:X^{k}\rightarrow X$
that is offered in this paper has the natural interpretation of representing
$k$th order recurrence relations of the form 
\begin{equation}
a_{n+k}=f(a_{n},a_{n+1},...,a_{n+k-1}).\label{eq:3}
\end{equation}
This representation is done by defining $f^{1}$ as a self-map on
$X^{k}$ that produces $k$ consecutive terms of the recurrence relation,
and defining iterates of $f$ as typical iterates of $f^{1}$. A formal
definition with examples is presented in section \ref{sec2}.

A self-map $f:X\rightarrow X$ is said to be \textit{involutory} (or
is an \textit{involution}) if it is its own inverse: $f^{2}=id_{X}$
(\citet{Aczel1948}). The immediate generalization of this property
is that a self-map's $(n-1$)th iterate is identical to its own inverse
for positive integer $n$:
\begin{equation}
f^{n}=id_{X}.\label{eq:3-1}
\end{equation}
Treating $f$ as an unknown, (\ref{eq:3-1}) defines a functional
equation known as \textit{Babbage's functional equation} (\citet{Babbage1815},
\citet{Babbage1816}, \citet{Babbage1820}, \citet{Babbage18211822})
and its solutions have been well studied in the literature. \citet{Lojasiewicz1951}
provides the general construction of the solution (see also \citet{Bogdanov1961},
and for certain real solutions, see the earlier work \citet{Ritt1916}).
Much of the work of the twentieth century relating to Babbage's equation---and
more broadly, the theory of functional equations---can be found in
the monograph \citet{Kuczma1990} (c.f. also \citet{Baron2001} for
a further development of this work).

A solution $f$ of (\ref{eq:3-1}) is generally referred to as an
\textit{nth iterative root of identity} (\citet{Kuczma1990}), and
when $n$ is the smallest positive integer such that $f$ satisfies
(\ref{eq:3-1}), $f$ is variously known as a \textit{function of
order $n$} or said to \textit{circulate with period $n$} (\citet{Ritt1916}),
or is\textit{ }said to be \textit{periodic }with\textit{ period $n$}
(\citet{McShane1961}). In this paper, we introduce and mostly use
the terminology \textit{involutory of order }$n$, or\textit{ $n$-involutory.}
As defined in section \ref{sec3}, a function $f:X^{k}\rightarrow X$
is said to be \textit{$n$-involutory} when its $n$th iterate is
the identity map. That is, $f$ is $n$-involutory when $f^{1}$,
a self-map on $X^{k}$, satisfies Babbage's equation throughout its
domain, $X^{k}$. More generally, a function $f:X^{k}\rightarrow X$
is \textit{$n$-involutory at point $x$} when $f^{1}$ satisfies
Babbage's equation at the specified point $x\in X^{k}$.\footnote{This terminology may seem redundant, although there are some subtle
distinctions between being $n$-involutory and being periodic with
period $n$. To begin with, I apply the term ``$n$-involutory''
to describe $f$, which need not be a self-map (if $k>1$), since
the focus of the paper is on a sufficient condition on $f$ under
which the self-map $f^{1}$ satisfies Babbage's equation. I also do
not restrict the order $n$ to be minimal in stating that a function
is $n$-involutory. More conceptually, however, the term ``$n$-involutory''
reminds the reader that the property it signifies is ultimately a
generalization of being involutory, and further, as will be made clear
in Proposition \ref{Prop2}, that involutory self-maps on $X$ directly
give rise to functions $f:X^{k}\rightarrow X$ having this property.} 

In section \ref{sec3.1}, I discuss elementary group-theoretic properties
of self-maps that are $n$-involutory. In this context, I also provide
a concise proof that continuous self-maps on $\mathbb{R}$ that are
involutory of an integral order must in fact be involutory. This latter
result, Proposition \ref{Prop1}, is well-known in the literature
on functional equations (c.f. \citet{Ewing1953}, \citet{Vincze1959},
\citet{McShane1961}), and is stated in Theorem 11.7.1 in \citet{Kuczma1990}.
Proposition \ref{Prop1} supplies general motivation in the sequel
for approaching the question of when multivariate functions $f:X^{k}\rightarrow X$
have a common involutory order (only possibly depending on $k$).
In section \ref{sec3.2}, I consider elementary properties of functions
$f:X^{k}\rightarrow X$ that are $n$-involutory. In the spirit of
Proposition \ref{Prop1}, the main result obtained in this section
is Proposition \ref{Prop2}, which asserts that when multivariate
functions are involutory in each of their $k$ arguments (a property
I refer to as being \textit{induced involutory}, defined more precisely
in the sequel), they are involutory of common order $k+1$. Section
\ref{sec3.3} discusses the relation between functions that are $n$-involutory
at a point and cycles of the recurrence relations that they represent.
Section \ref{sec4} concludes with discussion on possible extensions
and some research questions.

\section{{\large{}\label{sec2}Definition of iteration of functions $f:X^{k}\rightarrow X$}}

Consider the recurrence relation in (\ref{eq:3}). Observe that the
state of the system is characterized by a $k$-tuple of the elements
of the sequence itself. Our definition of function iteration for functions
$f:X^{k}\rightarrow X$ is based on representing such systems and
is given as follows:
\begin{defn}
\label{Defn1}For a set $X$ and function $f:X^{k}\rightarrow X$,
the $n$th iterate of $f$ , denoted by $f^{n}$ for nonnegative integer
$n$, is defined as the $n$th iterate of the self-map $f^{1}:X^{k}\rightarrow X^{k}$,
given by 
\[
f^{1}:(x_{1},x_{2},...,x_{k})\rightarrow(f_{1}^{1},f_{2}^{1},...,f_{k}^{1}),
\]
\[
f_{1}^{1}=f(x_{1},x_{2},...,x_{k}),
\]
\[
f_{2}^{1}=f(x_{2},x_{3},...,x_{k},f_{1}^{1}),
\]
\[
\vdots
\]
\[
f_{j}^{1}=f(x_{j},x_{j+1},...,x_{k},f_{1}^{1},f_{2}^{1},...,f_{j-1}^{1})
\]
\[
\vdots
\]
\[
f_{k}^{1}=f(x_{k},f_{1}^{1},f_{2}^{1},...,f_{k-1}^{1}).
\]
\end{defn}
Note that while $f$ itself is not a self-map, its iterates thus defined
are self-maps and satisfy the addition and multiplication rules, (\ref{eq:1})
and (\ref{eq:2}). This definition hence also allows for an immediate
extension to negative and fractional iterates in the way described
before. In the event the function $f$ is a self-map ($k=1$), this
definition naturally concurs with the usual definition of function
iteration. The graphical construction of $f^{1}$ for $f:\mathbb{R}^{2}\rightarrow\mathbb{R}$
is illustrated in Figure (\ref{fig1}). 

\begin{figure}[t]
\begin{centering}
\includegraphics[scale=0.25]{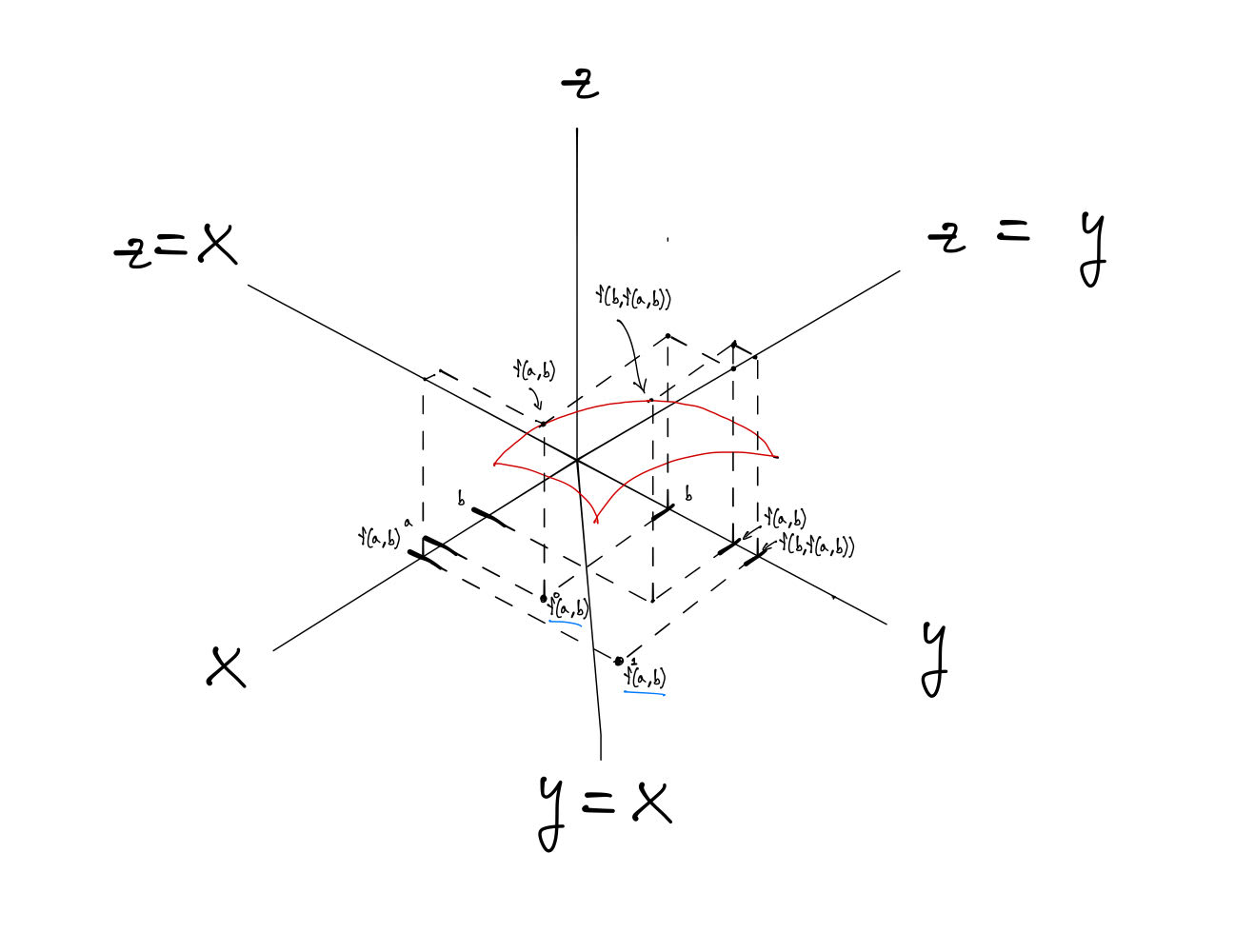}
\par\end{centering}
\caption{\label{fig1}Zeroth and first iterates (underlined in blue) at point
$(a,b)\in\mathbb{R}^{2}$ of function (red) $f:\mathbb{R}^{2}\rightarrow\mathbb{R}$}
\end{figure}
Definition (\ref{Defn1}) allows us to represent recurrence relations
described by (\ref{eq:3}) as an iterated function:
\begin{equation}
(a_{nk+1},a_{nk+2},...,a_{nk+k})=f^{n}(a_{1},a_{2},...,a_{k}).\label{eq:4}
\end{equation}

\subsection{\label{sec2.1}Examples}
\begin{example}
\label{Ex1}Consider the function $f:\mathbb{C}^{2}\rightarrow\mathbb{C}$
given by 
\[
f(x_{1},x_{2})=x_{1}+x_{2}.
\]
Its iterates encode terms of the recurrence relation $a_{n+2}=a_{n}+a_{n+1},$
and are given by 
\[
f^{n}(x_{1},x_{2})=(F_{2n-1}x_{1}+F_{2n}x_{2},F_{2n}x_{1}+F_{2n+1}x_{2}),
\]
where $F_{n}$ is the $n$th Fibonacci number, where $F_{0}=0$ and
$F_{1}=1.$
\end{example}
\begin{example}
\label{Ex2}Consider the function $f:X^{k}\rightarrow X$ for some
set $X$ given by 
\[
f(x_{1},x_{2},...,x_{k})=g(x_{j}),\quad j\in\{1,2,...,k\}
\]
for some function $g:X\rightarrow X$. The iterates of $f$ encode
terms of the recurrence relation $a_{n+k}=g(a_{n+j-1})$. In the case
$j=1$, the iterates are given by
\[
f^{n}(x_{1},x_{2},...,x_{k})=(g^{n}(x_{1}),g^{n}(x_{2}),..,g^{n}(x_{k})),
\]
while in the case $j=k$, the iterates (for positive integer $n$)
are given by 
\[
f^{n}(x_{1},x_{2},...,x_{k})=(g^{nk-k+1}(x_{k}),g^{nk-k+2}(x_{k}),..,g^{nk}(x_{k})).
\]
\end{example}
\begin{example}
\label{Ex3}Consider the function $f:\mathbb{C}^{k}\rightarrow\mathbb{C}$
given by 
\[
f(x_{1},x_{2},...,x_{k})=A-\sum_{j=1}^{k}x_{j},
\]
for some constant $A\in\mathbb{C}$, whose iterates encode terms of
the recurrence relation $a_{n+k}=A-\sum_{i=0}^{k-1}a_{n+i}$. Its
iterates are cyclical, such that for any integer $m$, they are given
by
\[
f^{m(k+1)}(x_{1},x_{2},...,x_{k})=(x_{1},x_{2},...,x_{k})
\]
\[
f^{m(k+1)+1}(x_{1},x_{2},...,x_{k})=(A-\sum_{j=1}^{k}x_{j},x_{1},x_{2},...,x_{k-1})
\]
\[
\vdots
\]
\[
f^{m(k+1)+p}(x_{1},x_{2},...,x_{k})=(x_{k-p+2},x_{k-p+3},...,x_{k},A-\sum_{j=1}^{k}x_{j},x_{1},x_{2},...,x_{k-p})
\]
\[
\vdots
\]
\[
f^{m(k+1)+k}(x_{1},x_{2},...,x_{k})=(x_{2},x_{3},...,x_{k},A-\sum_{i=1}^{k}x_{i}).
\]
\end{example}
Cycles such as those in the third example supply the motivation for
exploring functions $f:X^{k}\rightarrow X$ featuring such iterative
periodicity, as pursued in section \ref{sec3}.

\section{{\large{}\label{sec3}Periodicity}}

\subsection{\label{sec3.1}Periodicity of self-maps}

Based on the notion of function iteration established in Definition
(\ref{Defn1}), I offer the following definition as a generalization
of involutory functions:
\begin{defn}
\label{Defn2}A function $f:X^{k}\rightarrow X$ is \textit{involutory
of order $n$,} or \textit{$n$-involutory}, for integer $n$ when
$f^{n}=id_{X^{k}}.$ The integer $n$ is referred to as an \textit{involutory
order} of $f$.
\end{defn}
A function that is $n$-involutory is thus one whose $(n-1)$th iterate
is the inverse map of its first iterate. More generally, such a function
can be characterized as one whose $j$th iterate is the inverse map
of its $(n-j)$th iterate for any integer $j$. Note that our definition
is based on integral involutory orders to ensure well-definedness.
Of course, every function is $0$-involutory. If $f$ has a positive
involutory order, then when $n$ is the smallest positive involutory
order of $f$ (equivalently, when $f^{1}$ is periodic with period
$n$ as per \citet{McShane1961}), the set
\[
\{id_{X^{k}},f^{1},f^{2},...,f^{n-1}\},
\]
endowed with the composition operation, $\circ$, comprises a cyclical
group of order $n$ generated by $f^{1}$.

The following are immediate properties of such functions when they
are self-maps: 
\begin{lem}
\label{Lemma1} Let $f:X\rightarrow X$ be $n$-involutory \textit{for
$n\in\mathbb{Z}$. Then we have the following:}

1. $f$ is $nm$-involutory for any $m\in\mathbb{Z}$.

2. For $m\in\mathbb{Z}$, $f^{m}$ is $\frac{n}{gcd(n,m)}$-involutory. 

3. The identity map, $id_{X}$, is $m$-involutory for any $m\in\mathbb{Z}$.

4. If $f$ is also $m$-involutory for $m\in\mathbb{Z}$, and $gcd(m,n)=1$,
then $f=id_{X}$. 

5. \textit{For any invertible function $g:Y\rightarrow X$ }, the
map $g^{-1}\circ f\circ g$ is $n$-involutory.
\end{lem}
The properties of Lemma \ref{Lemma1} are apparent from the aforementioned
group structure. The fifth property is a conjugacy property that has
been noted as early as Babbage's work; in particular, when $Y=X$,
it reveals that any self-map on set $X$ that is $n$-involutory gives
rise to a conjugacy class of other functions that are also $n$-involutory
within the symmetric group on $X$. For example, the involutory self-map
$h$ defined on $X\equiv\mathbb{R}_{>0}$, given by $h:x\rightarrow\ln\left(\frac{e^{x}+1}{e^{x}-1}\right)$,
is conjugate to the linear self-map $-id_{X}$ since $h=g^{-1}\circ(-id_{X})\circ g$,
where $g=g_{1}\circ g_{2}\circ g_{3}\circ g_{4},$ with $g_{1}:x\rightarrow x-\frac{1}{2},g_{2}:x\rightarrow\log_{2}(x),g_{3}:x\rightarrow x-1,g_{4}:x\rightarrow e^{x}$.
In fact, it turns out all strictly decreasing involutions on a real
interval are conjugate to the negative identity map and hence are
conjugate to each other (c.f. Theorem 11.7.3 in \citet{Kuczma1990}).
Another result, perhaps more surprising, is that all complex rational
self-maps that are $n$-involutory fall into one of three simple explicit
conjugacy classes (c.f. Theorem 11.7.4 in \citet{Kuczma1990}). Exploring
this conjugacy property, and particularly determining when a function
is \textit{linearizable} (or conjugate to a linear map), has given
rise to a fruitful area of research within the theory of functional
equations (see, for instance, the recent work \citet{HOMSDONES2020},
which gives a review of several main results in this area for periodic
functions, and looks at linearization of functions satisfying a generalization
of Babbage's equation (\ref{eq:3-1})).

While Lemma \ref{Lemma1} applies to functions that are involutory
of arbitrary order $n$, for continuous self-maps on $\mathbb{R}$,
a function that is involutory of some integral order must be involutory
of order 2 (i.e. involutory), as the next proposition establishes.
\begin{prop}
\label{Prop1}Let $f:I\rightarrow I$ be a continuous self map on
interval $I\subset\mathbb{R}$. Then $f$ is $n$-involutory for some
integer $n$ if and only if either 1. $f=id_{I}$ or 2. $n$ is even
and $f$ is a strictly decreasing involution.
\end{prop}
\begin{proof}
\textcompwordmark{}

$(\Longleftarrow)$ Sufficiency is trivial by properties 1 and 3 of
Lemma \ref{Lemma1}.

$(\implies)$ Necessity is established in steps as follows:

Step 1: $f$ is strictly monotone

By continuity, it suffices to note that $f$ is injective, which follows
from the invertibility of $f$.

For the remaining steps, assume that $f\neq id_{I}$ i.e. $\exists z\in I$
s.t. $f(z)\gtrless z$; we will show $n$ is even and $f$ is a strictly
decreasing involution.

Step 2: $f$ is strictly decreasing

Assume by contradiction that $f$ is strictly increasing. Then $f(z)\gtrless z\implies f^{2}(z)\gtrless f(z)\gtrless z\implies...\implies z=f^{n}(z)\gtrless f^{n-1}(z)\gtrless...\gtrless f(z)\gtrless z$,
a contradiction.

Step 3: $n$ is even

Assume by contradiction $n$ is odd. By repeated application of step
2, we have $f(z)\gtrless z\implies z=f^{n}(z)\gtrless f^{n-1}(z)\lessgtr...\lessgtr f(z)\gtrless z=f^{n}(z)$.
The first inequality in this chain thus asserts that $z\gtrless f^{n-1}(z)$
, while the last inequality in this chain asserts $f(z)\gtrless f\left(f^{n-1}(z)\right)$,
contradicting that $f$ is strictly decreasing.

Step 4: $f^{2}=id_{I}$

Assume by contradiction that $\exists z'\in I$ s.t. $f^{2}(z')\gtrless f^{n}(z')=z'$.
By step 2, applying $f^{-1}$ to both sides of this inequality implies
$f(z')\lessgtr f^{n-1}(z')$. Applying the inverse again yields $f^{n}(z')=z'\gtrless f^{n-2}(z')$.
Repeatedly applying the inverse and noting $n$ is even (step 3) implies
$f^{2}(z')\gtrless f^{n}(z')\gtrless f^{n-2}(z')\gtrless...\gtrless f^{2}(z')$,
a contradiction. 
\end{proof}
As mentioned in section \ref{sec1}, the result given in Proposition
\ref{Prop1} is well established in the literature. Presenting the
result as such supplies some motivation in the sequel, where, in the
general spirit of this proposition, I obtain a sufficient condition
under which functions $f:X^{k}\rightarrow X$ are involutory of a
common order.

\subsection{\label{sec3.2}Periodicity of functions $f:X^{k}\rightarrow X$}

In this section, we consider properties of functions $f:X^{k}\rightarrow X$
that are $n$-involutory. A function $f$ being $n$-involutory is
equivalent to the self-map $f^{1}$ being $n$-involutory, so the
properties of Lemma \ref{Lemma1} extend quite naturally to such functions,
as stated in the following lemma: 
\begin{lem}
\label{Lemma2} Let $f:X^{k}\rightarrow X$ be \textit{$n$-involutory
for $n\in\mathbb{Z}$. Then we have the following:}

1. Properties 1-2 of Lemma \ref{Lemma1} hold.

2. The map $\hat{id}{}_{X^{k}}:X^{k}\rightarrow X$, defined as $\hat{id}{}_{X}:(x_{1},x_{2},...,x_{k})\rightarrow x_{1}$,
is $m$-involutory for any $m\in\mathbb{Z}$.

3. If $f$ is also $m$-involutory for $m\in\mathbb{Z}$, and $gcd(m,n)=1$,
then $f=\hat{id}{}_{X^{k}}$. 

4. For\textit{ $\tilde{f}:X\rightarrow X$ any involutory function,
}the map $\tilde{f}\circ\hat{id}{}_{X^{k}}$ is $2$-involutory.

5. \textit{For $g:Y\rightarrow X$ any invertible function,} let $\tilde{g}:Y^{k}\rightarrow X^{k}$
be given by
\[
\tilde{g}:(y_{1},y_{2},...,y_{k})\rightarrow(g(y_{1}),g(y_{2}),...,g(y_{k})).
\]

Then the map $g^{-1}\circ f\circ\tilde{g}:Y^{k}\rightarrow Y$ is
$n$-involutory.
\end{lem}
It is worth noting that properties 2-4 reveal that $\hat{id}{}_{X^{k}}$
acts as a kind of identity map for functions $f:X^{k}\rightarrow X$,
since $\hat{id}{}_{X^{k}}$ has its first iterate given by $id_{X^{k}}$.
However, a map $(x_{1},x_{2},...,x_{k})\rightarrow x_{j}\text{ for }j\neq1$
generally need not be involutory of an integral order. This contrast
results from the asymmetric manner in which the iterates of $f$ treat
their arguments. While the properties of Lemma \ref{Lemma1} thus
extend to functions $f:X^{k}\rightarrow X$, Proposition \ref{Prop1}
does not immediately extend in that continuous functions in Euclidean
space that are involutory of an integral order need not be involutory
of some common order (depending only possibly on $k$). For example,
consider the function $f:\mathbb{R}^{k}\rightarrow\mathbb{R}$ given
by $f:(x_{1},x_{2},...,x_{k})\rightarrow A-x_{1}$ for some constant
$A\in\mathbb{R}$, which is $2$-involutory by property 4 of Lemma
\ref{Lemma2}. In contrast, the function $f:\mathbb{R}^{k}\rightarrow\mathbb{R}$
given by $f:(x_{1},x_{2},...,x_{k})\rightarrow A-\sum_{i=1}^{k}x_{i}$
is $(k+1)$-involutory, as per Example \ref{Ex3}. However, Example
\ref{Ex3} displays certain properties that suggest sufficient conditions
under which one may obtain a notion of a common involutory order,
in the spirit of Proposition \ref{Prop1}. We define a few terms to
facilitate understanding these properties: 
\begin{defn}
\label{Defn3}Given function $f:X^{k}\rightarrow X$ , let $f_{j}(\cdot|x_{-j}):X\rightarrow X$
denote the induced function 
\[
f_{j}(x_{j}|x_{-j})\equiv f(x_{1},x_{2},...,x_{k})
\]
where $j\in\{1,2,...,k\}$ and $x_{-j}\equiv\{x_{1},x_{2},...,x_{k}\}\backslash\{x_{j}\}\in X^{k-1}$
is fixed. The function $f:X^{k}\rightarrow X$ is \textit{induced
involutory of order $n$ in argument $j$ }(written II-$n$\{$j$\})
when the induced function $f_{j}(\cdot|x_{-j})$ is $n$-involutory
for any $x_{-j}\in X^{k-1}$. The function is \textit{induced involutory
of order $n$ }(written II-$n$) when it is II-$n$\{$j$\} for every
$j\in\{1,2,...,k\}.$ The function is \textit{induced involutory in
argument $j$ }(written II-\{$j$\})\textit{ }when it is II-$2$\{$j$\}\textit{.
}The function is\textit{ induced involutory (II) }when it is II-$2$\textit{.}
\end{defn}
In other words, a function is II-$n$\{$j$\} when it is $n$-involutory
in argument $j$, holding fixed the other arguments. Also, recall
that a function $f:X^{k}\rightarrow X$ is said to be \textit{symmetric}
when its value is unchanged by any permutation of its $k$ arguments.
The following lemma establishes the relation between II-$n$ functions
and symmetry.
\begin{lem}
\label{Lemma3} If $f:X^{k}\rightarrow X$ is symmetric and II-$n\{j\}$
for integer $n$ and $j\in\{1,2,...,k\}$, then $f$ is II-$n$. If
$f:X^{k}\rightarrow X$ is II, then it is symmetric.
\end{lem}
\begin{proof}
The first claim is immediate since symmetry permits the arguments
of $f$ to be transposed so that for any $x_{1},x_{2},...,x_{k}\in X$
and any $j'\in\{1,2,...,k\}$, we have $x_{j}=f_{j}^{n}(x_{j}|x_{-j})=f_{j'}^{n}(x_{j}|x_{-j})$.

It suffices to show the second claim for $k=2$ since any permutation
of $k>2$ arguments is a composition of pairwise transpositions. Note
that $f$ being II implies $f$ is invertible with respect to each
argument. Observe that for any $x_{1},x_{2}\in X$, we have 
\[
x_{1}=f(f(x_{1},f(x_{2},x_{1})),f(x_{2},x_{1}))
\]
since $f$ is II-\{$1$\}, while we also have 
\[
x_{1}=f(x_{2},f(x_{2},x_{1}))
\]
since $f$ is II-\{$2$\}. Equating the two expressions for $x_{1}$
and applying the inverse of $f(\cdot,f(x_{2},x_{1}))$ to both implies
\[
x_{2}=f(x_{1},f(x_{2},x_{1})).
\]
We also have 
\[
x_{2}=f(x_{1},f(x_{1},x_{2}))
\]
since $f$ is II-\{$2$\}. Equating the two expressions for $x_{2}$
and applying the inverse of $f(x_{1},\cdot)$ to both implies $f(x_{1},x_{2})=f(x_{2},x_{1})$.
\end{proof}
Note that the second claim in Lemma \ref{Lemma3} does not apply to
II-$n$ functions for integers $n\geq3$. Consider the following example:
\begin{example}
\label{Ex4} Suppose $f:\mathbb{C}^{2}\rightarrow\mathbb{C}$ is given
by $f:(x_{1},x_{2})\rightarrow ax_{1}+bx_{2}$, where $a$ and $b$
are distinct $n$th roots of unity for integer $n\geq3$, neither
of which is unity itself. Then the $n$th iterates of $f$ with respect
to each argument are given as
\[
f_{1}^{n}(x_{1}|x_{2})=a^{n}x_{1}+b\frac{1-a^{n}}{1-a}x_{2}=x_{1},
\]
\[
f_{2}^{n}(x_{2}|x_{1})=b^{n}x_{2}+a\frac{1-b^{n}}{1-b}x_{1}=x_{2},
\]
so that $f$ is II-$n$ but not symmetric.
\end{example}
The following proposition asserts that II functions are involutory
of a common order.
\begin{prop}
\label{Prop2}If $f:X^{k}\rightarrow X$ is II, then $f$ is $(k+1)$-involutory.\footnote{In fact, $k+1$ is the minimal (positive integral) involutory order
of $f$, so that $f^{1}$ is periodic with period $k+1$.}
\end{prop}
\begin{proof}
Let $x$ denote $(x_{1},x_{2},...,x_{k})\in X^{k}$, and let $x_{-j}$
denote $\{x_{1},x_{2},...,x_{k}\}\backslash\{x_{j}\}\in X^{k-1}$.
By Lemma \ref{Lemma3}, $f$ is symmetric, so for any $j\in\{1,2,...,k\}$,
we have $f(f(x),x_{-j})=x_{j}$ (independent of the order of arguments).
Consequently, the first $k+1$ iterates are computed as follows: 
\[
f^{1}(x_{1},x_{2},...,x_{k})=(f(x),x_{1},x_{2},...,x_{k-1})
\]
\[
\vdots
\]
\[
f^{p}(x_{1},x_{2},...,x_{k})=(x_{k-p+2},x_{k-p+3},...,x_{k},f(x),x_{1},x_{2},...,x_{k-p})
\]
\[
\vdots
\]
\[
f^{k}(x_{1},x_{2},...,x_{k})=(x_{2},x_{3},...,x_{k},f(x))
\]

\[
f^{k+1}(x_{1},x_{2},...,x_{k})=(x_{1},x_{2},...,x_{k})
\]
\end{proof}
Proposition \ref{Prop2} thus permits a way to generate functions
that are involutory of any order $k+1$ by defining a function $f:X^{k}\rightarrow X$
that is a direct analogue of an involutory self-map on $X$, as in
Example \ref{Ex3}, which is the analogue of the self-map $f:x\rightarrow A-x.$
One can see how Proposition \ref{Prop2} breaks down when the II assumption
is violated. Consider the following examples:
\begin{example}
\label{Ex5}Let $f:\mathbb{C}^{2}\rightarrow\mathbb{C}$ be given
as $f:(x_{1},x_{2})\rightarrow ax_{1}+bx_{2}$, where $a=\frac{-1+i\sqrt{3}}{2}$
and $b=\frac{-1-i\sqrt{3}}{2}$. By Example \ref{Ex4}, $f$ is II-3,
but it is not involutory of any integral order, since its iterates
are given by 
\[
f^{n}(x_{1},x_{2})=(F_{2n-1}a^{n}x_{1}+F_{2n}a^{n+1}x_{2},F_{2n}a^{n+2}x_{1}+F_{2n+1}a^{n}x_{2}),
\]
where $F_{n}$ is the $n$th Fibonacci number, with $F_{0}=0$ and
$F_{1}=1.$
\end{example}
\begin{example}
\label{Ex6}Consider the symmetric function $f:X^{2}\rightarrow X$
specified below, where $X\equiv\{x_{1},x_{2},x_{3}\}$.

\begin{table}[H]
\begin{centering}
\begin{tabular}{c|c|c|c}
$f$ & $x_{1}$ & $x_{2}$ & $x_{3}$\tabularnewline
\hline 
\hline 
$x_{1}$ & $x_{1}$ & $x_{2}$ & $x_{3}$\tabularnewline
\hline 
$x_{2}$ & $x_{2}$ & $x_{3}$ & $x_{1}$\tabularnewline
\hline 
$x_{3}$ & $x_{3}$ & $x_{1}$ & $x_{2}$\tabularnewline
\end{tabular}
\par\end{centering}
\end{table}
The value of $f(x_{i},x_{j})$ for $(x_{i},x_{j})\in X^{2}$ is read
as the table element corresponding to row $x_{i}$ and column $x_{j}$.
It is easily verified that $f$ is II-3 since fixing any row or column,
the elements form a 3-cycle. For instance, the 3-cycle corresponding
to the third row is $x_{3}\rightarrow x_{2}\rightarrow x_{1}\rightarrow x_{3}$.
Moreover, $f$ is $4$-involutory since all the elements of the set
$X^{2}$ endowed with the self-map $f^{1}$ are generated as part
of two 4-cycles and one singleton cycle: 
\[
(x_{1},x_{2})\rightarrow f^{1}:(x_{2},x_{3})\rightarrow f^{2}:(x_{1},x_{3})\rightarrow f^{3}:(x_{3},x_{2})\rightarrow f^{4}:(x_{1},x_{2});
\]
\[
(x_{2},x_{1})\rightarrow f^{1}:(x_{2},x_{2})\rightarrow f^{2}:(x_{3},x_{1})\rightarrow f^{3}:(x_{3},x_{3})\rightarrow f^{4}:(x_{2},x_{1});
\]
\[
(x_{1},x_{1})\rightarrow f^{1}:(x_{1},x_{1}).
\]
\end{example}
\begin{example}
\label{Ex7}Consider the function $f:X^{2}\rightarrow X$ specified
below, where $X\equiv\{x_{1},x_{2},x_{3},x_{4}\}$.

\begin{table}[H]
\centering{}%
\begin{tabular}{c|c|c|c|c}
$f$ & $x_{1}$ & $x_{2}$ & $x_{3}$ & $x_{4}$\tabularnewline
\hline 
\hline 
$x_{1}$ & $x_{1}$ & $x_{3}$ & $x_{4}$ & $x_{2}$\tabularnewline
\hline 
$x_{2}$ & $x_{3}$ & $x_{1}$ & $x_{2}$ & $x_{4}$\tabularnewline
\hline 
$x_{3}$ & $x_{4}$ & $x_{2}$ & $x_{1}$ & $x_{3}$\tabularnewline
\hline 
$x_{4}$ & $x_{2}$ & $x_{4}$ & $x_{3}$ & $x_{1}$\tabularnewline
\end{tabular}
\end{table}

Here, $f$ is both symmetric and persymmetric (i.e. symmetric along
the antidiagonal) and II-3. Moreover, $f$ is $15$-involutory since
all the elements of the set $X^{2}$ endowed with the self-map $f^{1}$
are generated as part of a 15-cycle and singleton cycle:

\[
(x_{1},x_{2})\rightarrow f^{1}:(x_{3},x_{2})\rightarrow f^{2}:(x_{2},x_{1})\rightarrow f^{3}:(x_{3},x_{4})
\]
\[
\rightarrow f^{4}:(x_{3},x_{3})\rightarrow f^{5}:(x_{1},x_{4})\rightarrow f^{6}:(x_{2},x_{4})\rightarrow f^{7}:(x_{4},x_{1})
\]
\[
\rightarrow f^{8}:(x_{2},x_{3})\rightarrow f^{9}:(x_{2},x_{2})\rightarrow f^{10}:(x_{1},x_{3})\rightarrow f^{11}:(x_{4},x_{3})
\]
\[
\rightarrow f^{12}:(x_{3},x_{1})\rightarrow f^{13}:(x_{4},x_{2})\rightarrow f^{14}:(x_{4},x_{4})\rightarrow f^{15}:(x_{1},x_{2});
\]
\[
(x_{1},x_{1})\rightarrow f^{1}:(x_{1},x_{1}).
\]
\end{example}
Examples \ref{Ex6} and \ref{Ex7} show how functions $f:X^{k}\rightarrow X$
that are II-$n$ for $n\geq3$ may be involutory of some order, but
the involutory order is sensitive to the cardinality of $X$, unlike
II functions. 

\subsection{\label{sec3.3}Fixed points and recursive cycles}

A natural question to consider is what the interpretation is of functions
that are $n$-involutory in terms of the recurrence relations that
they represent. We can approach this question more generally in light
of a far less restrictive condition than being $n$-involutory:
\begin{defn}
A function $f:X^{k}\rightarrow X$ is \textit{$n$-involutory at point
$x\in X^{k}$ }when $f^{n}(x)=x$\textit{.}
\end{defn}
Functions that are $n$-involutory at some point reveal specific cycles
of recurrence relations that they represent. For instance, the function
in Example \ref{Ex3} is $(k+1)$-involutory, and relatedly, any seed
value in $\mathbb{C}^{k}$ generates a $(k+1)$-cycle of the recurrence
relation represented by it, $a_{n+k}=A-\sum_{i=0}^{k-1}a_{n+i}$.
This relationship is to be expected, since the iterates of a function
$f:X^{k}\rightarrow X$ encode every $k$ terms of the recurrence
relation it represents. The relationship is formalized in the following
claim: 
\begin{claim}
\label{Claim1} If a recurrence relation given by (\ref{eq:3}) has
a $j$-cycle, then $f$ is $\frac{j}{gcd(j,k)}$-involutory at a point.
If $f$ is $n$-involutory at a point, then the recurrence relation
has a $j$-cycle for some $j$ such that $j|n$.
\end{claim}
In the special case in which a recurrence relation given by (\ref{eq:3})
has a $k-$cycle characterized by some $x^{0}\in X^{k}$, then $f$
is $1$-involutory at $x^{0}$; that is, $x^{0}$ is a fixed point
of $f^{1}$. By definition, such a fixed point $x^{0}=(x_{1}^{0},x_{2}^{0},...,x_{k}^{0})\in X^{k}$
satisfies 
\[
f(x_{1}^{0},x_{2}^{0},...,x_{k}^{0})=x_{1}^{0}
\]
\[
f(x_{2}^{0},x_{3}^{0},...,x_{k}^{0},x_{1}^{0})=x_{2}^{0}
\]
\[
\vdots
\]
\[
f(x_{k}^{0},x_{1}^{0},x_{2}^{0},...,x_{k-1}^{0})=x_{k}^{0}.
\]
Conversely, such a fixed point of $f^{1}$ corresponds with a $j$-cycle
of the recurrence relation (\ref{eq:3}), where $j|k$. In the special
case in which the recurrence relation has a 1-cycle, so that the recurrence
relation is constant when seeded by the element $x\in X$ of the 1-cycle,
then the corresponding fixed point of $f^{1}$, $x^{0}\in X^{k}$,
is symmetric in the sense that $x_{i}^{0}=x$ $\forall i\in\{1,2,...,k\}$.
Of course, if $f$ is a symmetric function, then all the fixed points
of $f^{1}$ must be symmetric as such and the recurrence relation
it represents can only have 1-cycles. 

While Claim \ref{Claim1} asserts that $j$-cycles in a recurrence
relation correspond with its representative function $f:X^{k}\rightarrow X$
being involutory of some order at a point, in fact, any such $j$-cycle
can be understood as corresponding with a point at which a function
is $1$-involutory even when $j\nmid k$, by a trivial redefinition
of $f$ that augments its domain. By defining the map $\tilde{f}:X^{k'}\rightarrow X$
for $k'>k$ such that $j|k'$, given as 
\[
\tilde{f}:(x_{1},x_{2},...,x_{k'})\rightarrow f(\tilde{x}_{k'-k+1},\tilde{x}_{k'-k+2},...,\tilde{x}_{k'})|_{\tilde{x}_{m}=f(\tilde{x}_{m-k},\tilde{x}_{m-k+1},...,\tilde{x}_{m-1})\text{ for }m\geq k+1;\tilde{x}_{m}=x_{m}\text{ for }m\leq k},
\]
then a $j$-cycle of the recurrence relation represented by $f$ corresponds
with a fixed point of $\tilde{f}^{1}$. For instance, the recurrence
relation in Example \ref{Ex3} that was characterized by $f:\mathbb{C}^{k}\rightarrow\mathbb{C}$
---and for which every point in $\mathbb{C}^{k+1}$ is a $(k+1)$-cycle---
satisfies 
\[
a_{n+k+1}=\tilde{f}(a_{n},a_{n+1},...,a_{n+k}),
\]
where $\tilde{f}:\mathbb{C}^{k+1}\rightarrow\mathbb{C}$ is defined
as 
\[
\tilde{f}:(x_{1},x_{2},...,x_{k+1})\rightarrow A-\tilde{x}_{k+1}-\sum_{j=2}^{k}x_{m}|_{\tilde{x}_{k+1}=A-\sum_{q=1}^{k}x_{q}}=x_{1}.
\]
Thus, $\tilde{f}$ is simply $\hat{id}{}_{\mathbb{C}^{k+1}}$ (as
defined in Lemma \ref{Lemma2}), which is, of course, $1$-involutory
at every point in $\mathbb{C}^{k+1}$. In general, we can thus understand
attractors of a recurrence relation (\ref{eq:3}) through fixed point
iteration of the self-map $f^{1}$ given function $f$ (appropriately
augmented) that characterizes behavior of the recurrence relation.

\section{\label{sec4}Discussion}

This paper offers an introduction to a proposed notion of defining
function iteration for maps $f:X^{k}\rightarrow X$ that is based
on representing recurrence relations of the kind given in (\ref{eq:3}).
There are several questions that can be examined based on the definitions
posed in section \ref{sec3} that are not considered in this paper
but that can likely be addressed within the scope of combinatorics
and group theory. For instance, for finite sets $X$, one can consider
enumerating the possible functions that are II-$n$\{$j$\} for a
specified number of arguments $j$ (and possibly also involutory of
some integral order) and determining when they are symmetric. This
problem extends the line of research started in 1800 by Heinrich August
Rothe, who obtained a recursive formula defining the \textit{telephone
numbers}, which enumerate the involutions in the symmetric groups
(c.f. \citet{Chowla1951} and \citet{Knuth1973}). Moreover, as per
the fifth property of Lemma \ref{Lemma2}, one may consider the conjugacy
classes of functions that are involutory of some order, as well as
the class number associated with various sets.

Likewise, one may also consider how the second claim in Lemma \ref{Lemma3}
may extend to II-$n$ functions. We can heuristically consider a possible
extension by further examining Example \ref{Ex4}. Observe that the
function in Example \ref{Ex4} can afford to be asymmetric while still
satisfying the restriction of being II-3 because there are fewer than
three arguments ($k=2$). Suppose more generally that $f$ in this
setup is constrained to be II-$n$ for integer $n\geq2$ and is defined
as a linear map in $k$ arguments, i.e. $f:\mathbb{C}^{k}\rightarrow\mathbb{C}$,
where $f:(x_{1},x_{2},...,x_{k})\rightarrow\sum_{i=1}^{k}a_{i}x_{i}$.
Note that all coefficients $a_{i}$ must be $n$th roots of unity
but none can be unity itself, allowing the coefficients $n-1$ degrees
of freedom. However, if $k\geq m(n-1)+1$ for an integer $m\geq1$,
one can see that at least $m+1$ coefficients must be identical by
the pigeonhole principle. One can thus consider more generally if
a function $f:X^{k}\rightarrow X$ that is II-$n$ for integer $n\geq2$
with $k\geq m(n-1)+1$ must be symmetric in at least $m+1$ of its
arguments. 

Another question that may be considered is how Proposition \ref{Prop2}
might be extended. Examples \ref{Ex6} and \ref{Ex7} demonstrate
functions defined on finite sets that are II-$n$ (and symmetric)
and violate Proposition \ref{Prop2} insofar as not having a common
involutory order (depending only on $k$). Nonetheless, they are involutory
of some order. One can thus consider for finite sets $X$ under what
conditions functions that are II-$n$ are involutory of some order,
and how this order varies with the cardinality of $X$. One may also
consider how many distinct kinds of cycles are associated with such
functions (e.g. recall all elements in $X$ in Example \ref{Ex6}
were part of one of two 4-cycles or a singleton cycle). Moreover,
while Proposition \ref{Prop2} gives sufficient conditions for being
$n$-involutory, one may consider what the necessary conditions are,
including whether there are examples of such functions that are asymmetric
outside of the kind described by property 4 of Lemma \ref{Lemma2}.

A final point to consider is whether we may likewise suitably define
function iteration for functions $f:X^{k}\rightarrow X^{\ell}$. The
multiplicity of outputs may be reasonably treated as representing
distinct recursive equations, but there is no unique suitable way
to treat the inputs in this case. For instance, we may define function
iteration of a function $f:X^{3}\rightarrow X^{2}$ in a way that
represents the recursive system 
\[
a_{n+2}=f_{1}(a_{n},b_{n},a_{n+1})
\]
\[
b_{n+2}=f_{2}(a_{n},b_{n},b_{n+1}),
\]
or alternatively 
\[
a_{n+2}=f_{1}(a_{n},b_{n},a_{n+1})
\]
\[
b_{n+2}=f_{2}(a_{n},b_{n},a_{n+1}),
\]
or in terms of other combinations of three inputs. The state of the
first system is characterized by $(a_{n},a_{n+1},b_{n},b_{n+1})$
and function iterates may thus be defined as a self-map over $X^{4}$,
while the state of the second system is characterized by $(a_{n},a_{n+1},b_{n})$
and function iterates may thus be defined as a self-map over $X^{3}.$
Consequently, it would be more appropriate to define function iteration
in such cases in a way that suits the particular application considered.\newpage{}

\bibliographystyle{chicago}
\bibliography{Function_iteration_and_babbage_equation_linksv2}

\end{document}